\documentclass[a4paper,reqno]{amsart}

\usepackage[english]{babel}

\usepackage[T1]{fontenc}
\usepackage{lmodern}

\usepackage[tbtags]{mathtools}

\usepackage{amsfonts}
\usepackage{amssymb}
\usepackage{mathrsfs}

\usepackage[backend=bibtex,style=alphabetic,giveninits=true,maxnames=99]{biblatex}
\usepackage{csquotes}

\usepackage{enumitem}

\usepackage{tikz}
\usepackage{tikz-cd}

\usepackage{xurl}
\usepackage[linktocpage]{hyperref}
\usepackage{amsthm}
\usepackage[capitalize]{cleveref}

\hypersetup{breaklinks=true} 

\NewBibliographyString{toappear}
\DefineBibliographyStrings{english}{%
  toappear = {to appear},
}
\renewbibmacro*{in:}{%
  \iffieldundef{pubstate}
    {}
    {\printfield{pubstate}%
     \setunit{\addspace}%
     \clearfield{pubstate}\bibstring{in}\addspace}%
  \ifentrytype{article}{}{\bibstring{in}\printunit{\intitlepunct}}}

\renewbibmacro*{volume+number+eid}{\printfield{volume}\setunit*{\addcomma}\printfield{number}\setunit{\bibeidpunct}\printfield{eid}}
\DeclareFieldFormat[article]{number}{ no. #1} 

\DeclareFieldFormat{eprint:arXiv}{%
  arXiv\addcolon
  \ifhyperref
    {\href{https://arxiv.org/abs/#1}{%
       \nolinkurl{#1}%
       \iffieldundef{eprintclass}
         {}
         {\addspace\texttt{\mkbibbrackets{\thefield{eprintclass}}}}}}
    {\nolinkurl{#1}%
     \iffieldundef{eprintclass}
       {}
       {\addspace\texttt{\mkbibbrackets{\thefield{eprintclass}}}}}}

\newtheorem{thm}{Theorem}
\newtheorem{introthm}{Theorem}
\newtheorem{lemma}[thm]{Lemma}
\newtheorem{prop}[thm]{Proposition}
\newtheorem{cor}[thm]{Corollary}

\theoremstyle{definition}
\newtheorem{dfn}[thm]{Definition}
\newtheorem{ex}[thm]{Example}

\theoremstyle{remark}
\newtheorem{rem}[thm]{Remark}

\numberwithin{thm}{section}						
\numberwithin{equation}{section}				

\crefname{thm}{Theorem}{Theorems}
\Crefname{thm}{Theorem}{Theorems}
\crefname{introthm}{Theorem}{Theorems}
\Crefname{introthm}{Theorem}{Theorems}
\crefname{lemma}{Lemma}{Lemmas}
\Crefname{lemma}{Lemma}{Lemmas}
\crefname{prop}{Proposition}{Propositions}
\Crefname{prop}{Proposition}{Propositions}
\crefname{cor}{Corollary}{Corollaries}
\Crefname{cor}{Corollary}{Corollaries}
\crefname{dfn}{Definition}{Definitions}
\Crefname{dfn}{Definition}{Definitions}
\crefname{ex}{Example}{Examples}
\Crefname{ex}{Example}{Examples}
\crefname{rem}{Remark}{Remarks}
\Crefname{rem}{Remark}{Remarks}
\crefname{equation}{}{}
\Crefname{equation}{}{}
\crefname{enumi}{}{}
\Crefname{enumi}{}{}
\crefname{tfaei}{}{}
\Crefname{tfaei}{}{}

\setlist[itemize]{itemsep=0mm}
\setlist[enumerate]{itemsep=0mm}
\setlist[enumerate,1]{label=\arabic*)}	

\newlist{tfae}{enumerate}{1}
\setlist[tfae]{itemsep=0mm}
\setlist[tfae]{label=\Roman*)}

\newcommand{\ints}{\mathbb{Z}}

\newcommand{\<}{\langle}
\renewcommand{\>}{\rangle}

\newcommand{\kk}{\Bbbk}

\renewcommand{\:}{\colon}

\newcommand{\hookto}{\hookrightarrow}							
\newcommand{\tto}{\rightarrow\mathrel{\mkern-14mu}\rightarrow}	

\newcommand{\longto}[1]{\xrightarrow{#1}}						

\newenvironment{smallpmatrix}{\left(\begin{smallmatrix}}{\end{smallmatrix}\right)}

\newcommand{\svdots}{\scalebox{.75}{$\vdots$}}
\newcommand{\sddots}{\scalebox{.75}{$\ddots$}}

\newcommand{\cat}[1]{\mathscr{#1}}				
\newcommand{\sh}[1]{\mathcal{#1}}				

\DeclareMathOperator{\id}{id}					

\DeclareMathOperator{\gldim}{gldim}             

\DeclareMathOperator{\Tot}{Tot}					

\DeclareMathOperator{\tria}{tria}				
\DeclareMathOperator{\thick}{thick}				
\DeclareMathOperator{\add}{add}					
\DeclareMathOperator{\extclos}{extclos}			
\DeclareMathOperator{\Kar}{Kar}					

\newcommand{\rmodfd}{\mathbf{mod}_{\mathrm{fd}}\textrm{-}}		
\newcommand{\rprojfg}{\mathbf{proj}_{\mathrm{fg}}\textrm{-}}	

\newcommand{\grrMod}{\mathbf{Mod}^{\ints}\textrm{-}}					

\newcommand{\heart}{\heartsuit}							
\newcommand{\coheart}{\raisebox{\depth}{\scalebox{1}[-1]{$\heartsuit$}}}	

\newcommand{\C}{\mathbf{C}}                     

\newcommand{\D}{\mathbf{D}}						
\newcommand{\Db}{\mathbf{D}^\mathrm{b}}			
\newcommand{\Dfd}{\mathbf{D}_{\mathrm{fd}}}		

\newcommand{\K}{\mathbf{K}}						
\newcommand{\Kb}{\mathbf{K}^\mathrm{b}}			

\newcommand{\perf}{\mathbf{perf}}				

\newcommand{\Hom}{\mathrm{Hom}}					
\newcommand{\End}{\mathrm{End}}					
\newcommand{\Ext}{\mathrm{Ext}}					

\newcommand{\RHom}{\mathrm{RHom}}				
\newcommand{\REnd}{\mathrm{REnd}}				

\newcommand{\Lotimes}{\otimes^\mathrm{L}}		

\newcommand{\proj}{\mathbb{P}}					

\newcommand{\HH}{\mathrm{HH}}					

\newcommand{\dg}{\mathrm{dg}}					
\newcommand{\op}{\mathrm{op}}					

\title[Weight structures and t-structures generated by a $\proj$-object]{Weight structures and t-structures on the triangulated category generated by a $\proj$-object}

\author{Lukas Bonfert}
\address{L.B.: Leibniz University Hannover, Institute for Algebra, Number Theory and Discrete Mathematics (IAZD), Hannover, Germany}
\email{bonfert@math.uni-hannover.de}

\date{October 15, 2026}



\begin{document}
	\begin{abstract}
		$\proj$-objects were introduced by Huybrechts and Thomas as a generalization of spherical objects.
		We classify the weight structures and t-structures on the triangulated category generated by a $\proj^n[k]$-object in the sense of Krug.
		For $n=1$, this recovers the classification of weight structures and t-structures on the triangulated category generated by a spherical object due to Holm, J\o rgensen and Yang.
		We also compute the Koszul-dual model of the triangulated category generated by a $\proj^n[k]$-object.
	\end{abstract}
	\maketitle
	\section{Introduction}
		Weight structures and \textbf{t}-structures are standard tools for working with triangulated categories, with many applications in representation theory, algebraic geometry, and other areas.
		However, the classification of all weight structures and \textbf{t}-structures on a given triangulated category is usually not feasible.
		One interesting case where this has been achieved is the triangulated category generated by a spherical object.
		Spherical objects appear naturally in both algebraic geometry and representation theory, and the triangulated category they generate has been studied in many different settings, see the introduction of \cite{HolmJorgensenYang} for an overview.
		In \cite{HolmJorgensenYang} Holm, J\o rgensen and Yang classify the weight structures and \textbf{t}-structures on the triangulated category generated by a $k$-spherical object.
		More precisely, they show that for $k>0$ there is no non-trivial weight structure and a unique (up to shift) non-trivial \textbf{t}-structure, while for $k\leq 0$ there is no non-trivial \textbf{t}-structure and a unique (up to shift) non-trivial weight structure.
		
		A generalization of spherical objects are the $\proj$-objects introduced by Huybrechts and Thomas \cite{HuybrechtsThomas}.
		Following Krug \cite{Krug}, we call an object $P$ of a triangulated category $\cat{C}$ a \emph{$\proj^n[k]$-object} if its total endomorphism algebra $\End_\cat{C}^*(P)=\coprod_{r\in\ints} \Hom_\cat{C}(P,P[r])$ is isomorphic to $\kk[x]/(x^{n+1})$ with $|x|=k$ as a graded algebra, and if moreover $P$ is $nk$-Calabi--Yau.
		Hence a $\proj^1[k]$-object is the same as a $k$-spherical object.
		The main feature of $\proj$-objects and spherical objects is that they give rise to interesting autoequivalences known as $\proj$-twists \cite{HuybrechtsThomas}, respectively spherical twists \cite{SeidelThomas}.
		
		$\proj$-objects occur most naturally in algebraic geometry, see for instance \cite[Ex.~1.3]{HuybrechtsThomas} and the survey in \cite[\S 4.2]{Addington}.
		In particular, by \cite[Observation~1.2]{Krug} a compact K\"ahler manifold $X$ of dimension $2n$ is hyperk\"ahler if and only if $\sh{O}_X$ is a $\proj^n[2]$-object in $\D(X)$.
		Varieties whose structure sheaves are $\proj^n[k]$-objects with $k>2$ can be constructed from so-called strict Enriques varieties, see \cite[Thm.~1.3]{Krug}.
		In particular these exist for $n=2$ and $n=3$, with any even $k>2$.
		
		$\proj^n[k]$-objects also appear in representation theory, but so far they have rarely been studied from this perspective.
		By \cite[Thm.~B]{PerverseSheavesPaper} all indecomposable perverse sheaves on projective space are $\proj^n[2]$-like objects (i.e.~their total endomorphism algebras are $\kk[x]/(x^{n+1})$, but most of these objects are not Calabi--Yau).
		Furthermore, it is easy to see that the free dg module of rank~$1$ over the dg algebra $A=\kk[x]/(x^{n+1})$ with $|x|=k$ and trivial differential is a $\proj^n[k]$-object in $\perf(A)$.
		This is actually the ``universal'' $\proj^n[k]$-object: as a consequence of the formality result from \cite{HocheneggerKrugFormality}, if $P$ is a $\proj^n[k]$ in an algebraic triangulated category $\cat{C}$ then $\Hom_\cat{C}(P,-)\: \thick_\cat{C}(P)\to\perf(A)$ is an equivalence taking $P$ to $A$.
		In particular, the triangulated category generated by a $\proj^n[k]$-object is independent of the ambient triangulated category, see \cref{triangcatpobj} for details.
		
		Our main result classifies the weight structures and \textbf{t}-structures on the triangulated category generated by a $\proj^n[k]$-object.
		\begin{introthm}[\cref{classificationthm}]
			\label{classificationthmintro}
			Let $\cat{T}$ be the triangulated category generated by a $\proj^n[k]$-object.
			\begin{enumerate}
				\item For $n=0$, $\cat{T}$ admits a standard weight structure and a standard \textbf{t}-structure.
					Any non-trivial \textbf{t}-structure or weight structure on $\cat{T}$ is a shift of these.
				\item For $n>0$ and $k>0$, $\cat{T}$ admits a standard \textbf{t}-structure, and any non-trivial \textbf{t}-structure on $\cat{T}$ is a shift of this.
					$\cat{T}$ does not admit a non-trivial weight structure.
				\item For $n>0$ and $k\leq 0$, $\cat{T}$ admits a standard weight structure, and any non-trivial weight structure on $\cat{T}$ is a shift of this.
					$\cat{T}$ does not admit a non-trivial \textbf{t}-structure.
			\end{enumerate}
		\end{introthm}
		For $n=0$ the classification is well-known since in this case $\cat{T}$ is just $\Db(\rmodfd\kk)$.
		For $n=1$, \cref{classificationthmintro} recovers \cite[Thm.~A]{HolmJorgensenYang}.
		
		The proof of \cref{classificationthmintro} occupies most of \cref{classification}.
		Standard results from silting theory provide the standard weight structure (for $k\leq 0$) respectively \textbf{t}-structure (for $k>0$), see \cref{standardstructs}.
		The main idea for the classification is to show that if $X\in\cat{T}$ does not lie in the standard (co)heart up to shift, then $\Hom_\cat{T}(X,X[r])\neq 0$ for some $r>0$ (if $k\leq 0$) respectively $r<0$ (if $k>0$).
		From this it follows that any non-zero (co)heart must already contain a shift of $P$.
		The required morphisms are constructed separately for $k<0$, $k=0$, and $k>0$, although the arguments for $k<0$ and $k=0$ are very similar.
		For $k>0$ our argument is ``model-free'' and only uses the standard \textbf{t}-structure, while for $k\leq 0$ we explicitly use that $\cat{T}$ is equivalent to $\perf(A)$ where $A=\kk[x]/(x^{n+1})$ with $|x|=k$ and trivial differential.
		Finally, \textbf{t}-structures and weight structures with zero (co)heart are easily seen to be trivial, see \cref{zeroheart}.
		The non-existence of \textbf{t}-structures respectively weight structures follows from this and a well-known obstruction in terms of the the Calabi--Yau dimension, see \cref{nonexistence}.
		
		The fundamental idea of the proof (i.e.~constructing morphisms that obstruct objects other than $P$ from lying in (co)hearts) is the same as in \cite{HolmJorgensenYang}.
		However, these morphisms are obtained differently.
		The approach in \cite{HolmJorgensenYang} is to explicitly describe the entire triangulated category $\cat{T}$ by writing down the Auslander--Reiten quiver of $\cat{T}$.
		The non-zero Hom spaces are computed via long exact sequences, and described combinatorially in terms of the Auslander--Reiten quiver.
		The classification of weight structures and \textbf{t}-structures can then essentially be read off.
		Instead of the ``canonical'' model $\cat{T}=\perf(\kk[x]/(x^2))$ with $|x|=k$, these arguments use the Koszul-dual model: $\cat{T}$ is equivalent to $\thick(\kk)\subseteq\D(\kk[t])$, where $|t|=1-k$.
		This allows to explicitly list all indecomposable objects, as $\kk[t]$ is graded hereditary.
		
		This approach does not work in general, as $\perf(\kk[x]/(x^{n+1})$ is wild for $n>1$.
		Nevertheless some things can be said about the Auslander--Reiten quiver, at least for $k>1$: by \cite[Thm.~4.2]{JorgensenARQuiverPoincare} every component is of the shape $\ints A_\infty$ (just as for $n=1$), and by \cite[Thm.~1]{Schmidt} there are infinitely many connected components.
		In particular, the Auslander--Reiten quiver again has the same shape as that of $\Kb(\rprojfg A)$ for the ungraded algebra $A$, cf.~\cite{Wheeler} and \cite{Webb}.
		
		In \cref{koszulduality} we explicitly determine the Koszul-dual model of $\cat{T}$ for $n>1$, which is provided by the general theory of dg Koszul duality from \cite{KellerDerivingDGCats}.
		In general, the dg Koszul dual of an augmented dg algebra $A$ is defined as $A^{!,\dg}=\RHom_A(\kk,\kk)$.
		Our second main result is the following explicit description of $A^{!,\dg}$ for the dg algebra describing the triangulated category generated by a $\proj^n[k]$-object.
		\begin{introthm}[\cref{koszuldual}]
			\label{koszuldualintro}
			Let $A=\kk[x]/(x^{n+1})$ with $|x|=k$ and $n>1$, viewed as a dg algebra with trivial differential.
			\begin{enumerate}
				\item $A^{!,\dg}$ is quasi-isomorphic to the dg algebra
					\begin{equation*}
						\kk\<s,h_0,\dots,h_{n-2}\>/(h_ih_j,s^2h_i-h_is^2, h_{i+1}sh_j-h_ish_{j+1}, s^2h_i-h_{n-2}sh_{i+1})
					\end{equation*}
					with the grading given by $|s|=1-k$ and $|h_i|=1-(n-i)k$, and the differential defined on generators by $d(s)=0$, $d(h_{n-2})=s^2$ and $d(h_i)=sh_{i+1}+h_{i+1}s$ for $i<n-2$.
				\item We have $H^*(A^{!,\dg})\cong\kk[s,t]/(s^2)$, where $t=sh_0+h_0s$ is of degree $|t|=2-(n+1)k$.
				\item $A^{!,\dg}$ is not formal.
			\end{enumerate}
		\end{introthm}
		The computation of the dg Koszul dual is rather technical.
		The crucial \cref{dgvsgradedext} is a variant of \cite[Lemma~5.3]{KalckYangGentle} for diagonal-finite unbounded complexes.
		In our setting, this applies to the projective resolution $P_\kk$ of $\kk$ as graded $A$-module unless $n=1$ and $k=1$.
		From this it follows that $A^{!,\dg}$ is isomorphic to the total complex of the $\ints$-graded dg algebra $\End_{\C_\dg(\grrMod A)}^*(P_\kk)$, unless $n=1$ and $k=1$.
		
		For the proof of \cref{koszuldualintro} we work with a $\ints$-graded version of the claimed dg Koszul dual.
		A straightforward but tedious computation shows that its cohomology is $\kk[s,t]/(s^2)$ with $t=sh_0+h_0s$.
		On the other hand, it is easy to write down generators for $H^*(\End_{\C_\dg(\grrMod A)}^*(P_\kk))=\Ext_{\grrMod A}^{*,*}(\kk,\kk)\cong\kk[s,t]/(s^2)$, see \cref{extalgebra}.
		By sending the generators $s$ and $h_i$ to explicit chain maps and homotopies we thus obtain the desired quasi-isomorphism of $\ints$-graded dg algebras, and \cref{koszuldualintro} follows by taking the total complex.
		To see that $A^{!,\dg}$ is not formal we show that certain triple Massey products do not vanish.
		
		For $n=1$ and $k\neq 1$, similar but easier arguments also yield $A^{!,\dg}=\kk[s]$ with $|s|=1-k$, see \cref{koszuldualn1}.
		For $n=1$ and $k=1$, we compute $A^{!,\dg}$ directly to obtain $A^{!,\dg}=\kk[\![t]\!]$ with $|t|=0$, see \cref{koszuldualn1k1}.
		The results for $n=1$ are implicitly stated in \cite[\S 2]{KellerYangZhou} and presumably well-known to the experts.
		
		In \cref{pvdfinite} we show that $\kk[x]/(x^{n+1})$ with $|x|=k>0$ is pvd-finite in the sense of \cite{FushimiDgKoszul} unless $n=1$ and $k=1$.
		From this it follows that the dg Koszul dual provides an alternative model for the triangulated category generated by a $\proj^n[k]$-object, see \cref{koszulmodel}.
		\subsection{Acknowledgements}
			I would like to thank Andreas Hochenegger, Thorsten Holm, Peter J\o rgensen and Andreas Krug for helpful discussions.
	\section{The triangulated category generated by a \texorpdfstring{$\proj^n[k]$}{Pn[k]}-object}
		\subsection{\texorpdfstring{$\proj^n[k]$}{Pn[k]}-objects}
			Let $\cat{C}$ be a $\kk$-linear triangulated category.
			For $X\in\cat{C}$ we denote its total endomorphism algebra by $\End_\cat{C}^*(X)=\coprod_{n\in\ints} \Hom_\cat{C}(X,X[n])$.
			This is $\ints$-graded by $|f|=n$ for $f\: X\to X[n]$.
			
			We begin by recalling the definition of $\proj^n$-objects introduced in \cite{HuybrechtsThomas}, using the slight generalization to $\proj^n[k]$-objects from \cite{Krug}.
			Recall that an object $X\in\cat{C}$ is called $d$-Calabi--Yau if $\Hom_\cat{C}(X,-)\cong\Hom_\cat{C}(-,X[d])^\vee$, where $(-)^\vee=\Hom_\kk(-,\kk)$.
			\begin{dfn}
				Let $n\geq 0$ and $k\in\ints$.
				\begin{itemize}
					\item An object $P\in\cat{C}$ is \emph{$\proj^n[k]$-like} if there is an isomorphism of graded algebras $\End_\cat{C}^*(P)\cong\kk[x]/(x^{n+1})$ with $|x|=k$.
					\item $P$ is a \emph{$\proj^n[k]$-object} if moreover $P$ is $nk$-Calabi--Yau.
				\end{itemize}
			\end{dfn}
			\begin{ex}
				The following are well-known special cases:
				\begin{enumerate}
					\item For $n=0$, the value of $k$ is undefined.
						A $\proj^0$-like object is the same as an exceptional object.
					\item For $n=1$, a $\proj^1[k]$-like object is the same as a $k$-spherelike object in the sense of \cite{HocheneggerKalckPloog}.
						A $\proj^1[k]$-object is the same as a $k$-spherical object in the sense of \cite{SeidelThomas}.
					\item For $k=2$ we recover the original definition of $\proj^n$-objects from \cite{HuybrechtsThomas}.
				\end{enumerate}
			\end{ex}
		\subsection{The triangulated category}
			\label{triangcatpobj}
			In this subsection we show that the triangulated category generated by a $\proj^n[k]$-object is independent of the ambient triangulated category.
			For $k>0$ this is due to Hochenegger and Krug \cite{HocheneggerKrugFormality}, and the general proof is the same.
			
			A graded algebra $A$ is called \emph{intrinsically formal} if any dg algebra $B$ with $H^*(B)\cong A$ as graded algebras is quasi-isomorphic to $A$.
			
			The following proposition is essentially \cite[Lemma~2.8]{HocheneggerKrugFormality}.
			We include the proof for completeness.
			\begin{prop}[Hochenegger--Krug]
				\label{intrinsicformality}
				For $n\geq 0$ and $k\in\ints$, the graded algebra $\kk[x]/(x^{n+1})$ with $|x|=k$ is intrinsically formal.
			\end{prop}
			\begin{proof}
				For $n=0$ this is obvious: a dg algebra $B$ with $H^*(B)\cong\kk$ (concentrated in degree~$0$) is quasi-isomorphic to $\kk$ by truncation first to degrees $\leq 0$, and then to degree~$0$.

				For $n>0$, we follow the proof from \cite[Lemma~2.8]{HocheneggerKrugFormality}, which works for any $k$ but with slightly different estimates.
				Let $A=\kk[t]/(x^{n+1})$ with $|x|=k$.
				By Kadeishvili's criterion \cite[Cor.~4]{Kadeishvili} (see also \cite[Thm.~4.7]{SeidelThomas}), to prove intrinsic formality of $A$ it suffices to show that $\HH^{q,2-q}(A,A)=0$ for all $q>2$.
				
				Write $A^e=A\otimes_\kk A^\op$ for the enveloping algebra of $A$.
				A free resolution of $A$ as graded $A^e$-module (i.e.~of $A$ as graded $A$-$A$-bimodule) is given by
				\begin{equation*}
					F=(\dots\to A^e\<-2(n+1)k\>\longto{v}A^e\<-(n+2)k\>\longto{u}A^e\<-(n+1)k\>\longto{v}A^e\<-k\>\longto{u} A^e),
				\end{equation*}
				where $u=x\otimes 1-1\otimes x$ and $v=\sum_{i=1}^n x^i\otimes x^{n-i}$, cf.~\cite[Exercise~9.1.4]{Weibel}.
				Explicitly, we have for $q\geq 0$
				\begin{equation*}
					F^{-q}=\begin{cases}
						A^e\<-(n+1)k\frac{q}{2}\>&\text{if $q$ even},\\
						A^e\<-(n+1)k\frac{q-1}{2}-k\>&\text{if $q$ odd}.
					\end{cases}
				\end{equation*}
				By definition we have for $q\geq 0$
				\begin{equation*}
					\HH^{q,2-q}(A,A)\cong H^q(\Hom_{\grrMod A^e}(F,A\<2-q\>)),
				\end{equation*}
				which is a subquotient of $\Hom_{\grrMod A^e}(F^{-q},A\<2-q\>)$.
				Hence to apply Kadeishvili's criterion it suffices to show that $\Hom_{\grrMod A^e}(F^{-q},A\<2-q\>)=0$ for all $q>2$.
				This follows ultimately by degree reasons, but with several case distinctions.
				
				For $q=2i$ with $i>1$ we have
				\begin{equation*}
					\begin{split}
						\Hom_{\grrMod A^e}(F^{-q},A\<2-q\>)&=\Hom_{\grrMod A^e}(A^e\<-(n+1)ki\>,A\<2-2i\>)\\
						&\cong A^{2-2i+(n+1)ki}.
					\end{split}
				\end{equation*}
				Observe that $A$ is concentrated in degrees $0,k,\dots,nk$, from which it follows that $A^{2-2i+(n+1)ki}=0$ for all $k\in\ints$ and $i>1$:
				\begin{itemize}
					\item If $k>1$, then
						\begin{equation*}
							2-2i+(n+1)ki=2+(k-2)i+nki>2+(k-2)i+nk>nk.
						\end{equation*}
					\item If $k=1$, then
						\begin{equation*}
							2-2i+(n+1)ki=2+(n-1)i\geq 2+2(n-1)=2n>n=nk
						\end{equation*}
						since $n>0$ and $i>1$.
					\item If $k=0$, then $2-2i+(n+1)ki=2-2i<0$ since $i>1$.
					\item If $k<0$, then $2-2i+(n+1)ki<(n+1)ki<nk$ since $i>1$.
				\end{itemize}
				For $q=2i+1$ with $i\geq 1$ we similarly have
				\begin{equation*}
					\begin{split}
						\Hom_{\grrMod A^e}(F^{-q},A\<2-q\>)&=\Hom_{\grrMod A^e}(A^e\<-(n+1)ki-k\>,A\<1-2i\>)\\
						&\cong A^{1-2i+(n+1)ki+k},
					\end{split}
				\end{equation*}
				which also vanishes for all $k\in\ints$ and $i\geq 1$:
				\begin{itemize}
					\item If $k>1$, then
						\begin{equation*}
							1-2i+(n+1)ki+k=1+(k-2)i+k+nki>nki\geq nk.
						\end{equation*}
					\item If $k=1$, then
						\begin{equation*}
							1-2i+(n+1)ki+k=2+(n-1)i\geq 2+(n-1)>n=nk
						\end{equation*}
						since $n>0$ and $i\geq 1$.
					\item If $k=0$, then $1-2i+(n+1)ki+k=1-2i<0$ as $i\geq 1$.
					\item If $k<0$, then $1-2i+(n+1)ki+k<(n+1)ki<nk$.\qedhere
				\end{itemize}
			\end{proof}
			From \cref{intrinsicformality} it follows that the triangulated category $\cat{T}$ generated by a compact $\proj^n[k]$-like object is independent of the ambient triangulated category.
			More precisely, we obtain the following concrete model of $\cat{T}$.
			\begin{prop}
				\label{moritamodel}
				Let $\cat{C}$ be an idempotent-complete algebraic triangulated category, $P\in\cat{C}$ a $\proj^n[k]$-like object and $\cat{T}=\thick_\cat{C}(P)$.
				Then $\cat{T}\cong\perf(A)$, where $A=\kk[x]/(x^{n+1})$ is viewed as a dg algebra with $|x|=k$ and differential $d=0$.
			\end{prop}
			\begin{proof}
				By \cite[Thm.~3.8]{KellerICM}, there is an equivalence $\thick_\cat{C}(P)\cong\perf(\REnd_\cat{C}(P))$.
				By the definition of $\proj^n[k]$-like objects we have an isomorphism of graded algebras $H^*(\REnd_\cat{C}(P))\cong A$.
				Since $A$ is intrinsically formal by \cref{intrinsicformality}, there is a quasi-isomorphism $\REnd_\cat{C}(P)\simeq A$, and thus $\perf(\REnd_\cat{C}(P))\cong\perf(A)$.
			\end{proof}
			The following lemma shows that a $\proj^n[k]$-like object is a $\proj^n[k]$-object in the thick subcategory that it generates, cf.~\cite[Rem.~2.2]{HocheneggerKrugFormality}.
			The argument is the same used for spherelike objects in \cite[Prop.~2.8]{HolmJorgensenYang}, and we include the proof for completeness.
			\begin{lemma}
				\label{plikecy}
				The triangulated category $\cat{T}$ generated by a $\proj^n[k]$-like object is $nk$-Calabi--Yau.
				In particular $P$ is a $\proj^n[k]$-object in $\cat{T}$.
			\end{lemma}
			\begin{proof}
				By \cref{moritamodel} we have $\cat{T}\cong\perf(A)$, where $A=\kk[t]/(x^{n+1})$ with $|x|=k$.
				For $X,Y\in\perf(A)$, tensor-hom adjunction and $\kk$-linear duality provide natural isomorphisms
				\begin{equation*}
					\begin{split}
						\RHom_A(Y,D(A)\Lotimes_A X)^\vee&\cong \RHom_A(Y,D(A)\Lotimes_A X)^\vee\\
						&\cong\RHom_{A^\op}(\RHom_A(Y,D(A)),D(X)\\
						&\cong\RHom_{A^\op}(D(Y),D(X))\\
						&\cong\RHom_A(X,Y),
					\end{split}
				\end{equation*}
				where $D=\RHom_\kk(-,\kk)\: \D(A)\to\D(A^\op)$ denotes the $\kk$-linear duality functor.
				By taking $H^0$, we get $\Hom_{\perf(A)}(Y,D(A)\Lotimes_A X)^\vee\cong\Hom_{\perf(A)}(X,Y)$, and hence $D(A)\Lotimes_A-$ is a Serre functor for $\perf(A)$.
				The claim follows since $D(A)\cong A[nk]$ as dg $A$-$A$-bimodule.
			\end{proof}
			\begin{rem}
				If $P$ is a $\proj^n[k]$-object, one can also show that $\cat{T}=\thick_\cat{C}(P)$ is Calabi--Yau by a standard d\'evissage argument: the subcategory of objects $X$ for which $\Hom_\cat{C}(X,-)\cong\Hom_\cat{C}(-,X[nk])^\vee$ is thick and contains $P$.
			\end{rem}
	\section{Classification of weight structures and t-structures}
		\label{classification}
		In this section we classify the weight structures and \textbf{t}-structures on the triangulated category $\cat{T}$ generated by a $\proj^n[k]$-object.
		First we recall the definitions.
		\subsection{Weight structures and t-structures}
			\textbf{t}-structures on triangulated categories were introduced in \cite{BBD}.
			\begin{dfn}
				A \emph{\textbf{t}-structure} on a triangulated category $\cat{D}$ is a pair $t=(\cat{D}^{t\leq 0},\cat{D}^{t\geq 0})$ of strict full subcategories such that
				\begin{itemize}
					\item $\cat{D}^{t\leq 0}[1]\subseteq\cat{D}^{t\leq 0}$ and $\cat{D}^{t\geq 0}[-1]\subseteq\cat{D}^{t\geq 0}$,
					\item $\cat{D}^{t\leq 0}\perp\cat{D}^{t\geq 0}[-1]$,
					\item for all $X\in\cat{D}$ there is a triangle (called \emph{\textbf{t}-decomposition} of $X$)
						\begin{equation*}
							t_{\leq 0}X\to X\to t_{>0}X\to t_{\leq 0}X[1]
						\end{equation*}
						with $t_{\leq 0}X\in\cat{D}^{t\leq 0}$ and $t_{>0}X\in\cat{D}^{t\geq 0}[-1]$.
				\end{itemize}
				The full subcategory $\heart_t=\cat{D}^{t\leq 0}\cap\cat{D}^{t\geq 0}$ is called the \emph{heart} of $t$.
				$t$ is \emph{bounded} if $\cat{D}=\tria_\cat{D}(\heart_t)$.
			\end{dfn}
			We also use the notation $\cat{D}^{t>0}=\cat{D}^{t\geq 0}[-1]$ and $\cat{D}^{t<0}=\cat{D}^{t\leq 0}[1]$, and also $\cat{D}^{t\geq n}=\cat{D}^{t\geq 0}[-n]$ for $n\in\ints$ (and analogously $\cat{D}^{t\leq n}$).
			Recall that \textbf{t}-decompositions are unique up to isomorphism, and furthermore define functors $t_{>0}\: \cat{D}\to \cat{D}^{t>0}$ and $t_{\leq 0}\: \cat{D}\to\cat{D}^{t\leq 0}$ that are left (resp.~right) adjoint to the respective inclusions \cite[Prop.~1.3.3]{BBD}.
			By shifting we also obtain truncation functors $t_{\leq r}$, $t_{<r}$, $t_{\geq r}$ and $t_{>r}$ for any $r\in\ints$, with similar adjunction properties.
			
			The heart of a \textbf{t}-structure is an abelian category, with the short exact sequences $0\to X\to Y\to Z\to 0$ in $\heart_t$ in bijection with triangles $X\to Y\to Z\to X[1]$, cf.~\cite[Prop.~1.2.2]{BBD}.
			This gives rise to a natural isomorphism $\Hom_\cat{D}(-,-[1])\cong\Ext_{\heart_t}^1(-,-)$.
			
			Weight structures (also known as co-\textbf{t}-structures) were originally defined in \cite{Bondarko} and \cite{Pauksztello}.
			\begin{dfn}
				A \emph{weight structure} on a triangulated category $\cat{C}$ is a pair $w=(\cat{C}_{w\leq 0},\cat{C}_{w\geq 0})$ of Karoubi-closed full subcategories such that
				\begin{itemize}
					\item $\cat{C}_{w\leq 0}[1]\subseteq\cat{C}_{w\leq 0}$ and $\cat{C}_{w\geq 0}[-1]\subseteq\cat{C}_{w\geq 0}$,
					\item $\cat{C}_{w\geq 0}[-1]\perp\cat{C}_{w\leq 0}$,
					\item for all $X\in\cat{C}$ there is a triangle (called \emph{weight decomposition} of $X$)
						\begin{equation*}
							w_{>0}X\to X\to w_{\leq 0}X\to w_{>0}X[1]
						\end{equation*}
						with $w_{>0}X\in\cat{C}_{w\geq 0}[-1]$ and $w_{\leq 0}X\in\cat{C}_{w\leq 0}$.
				\end{itemize}
				The full subcategory $\coheart_w=\cat{C}_{w\leq 0}\cap\cat{C}_{w\geq 0}$ is called the \emph{coheart} of $w$.
				$w$ is \emph{bounded} if $\cat{C}=\thick_\cat{C}(\coheart_w)$.
			\end{dfn}
			As for \textbf{t}-structures we write $\cat{C}_{w>0}=\cat{C}_{w\geq 0}[-1]$, $\cat{C}_{w\geq n}=\cat{C}_{w\geq 0}[-n]$, and so on.
			In contrast to \textbf{t}-decompositions, by \cite[Rem.~1.2.2]{Bondarko} weight decompositions are not unique unless $\coheart_w=\{0\}$.
			In particular $w_{\leq 0}$ and $w_{>0}$ do not define functors.
		\subsection{The classification result}
			\label{standardstructs}
			Let $\cat{T}$ be the triangulated category generated by a $\proj^n[k]$-object $P$ for $n\geq 0$ and $k\in\ints$.
			Concretely, by \cref{moritamodel} we have $\cat{T}=\perf(A)$ where $A=\kk[x]/(x^{n+1})$ with $|x|=k$, and then the $\proj^n[k]$-object is $P=A$.
			Our goal is to classify the weight structures and \textbf{t}-structures on $\cat{T}$.
			
			Independent of the values of $n$ and $k$, we always have the two trivial weight structures and \textbf{t}-structures $(\cat{T},\{0\})$ and $(\{0\},\cat{T})$ (as on any triangulated category).
			
			Depending on the value of $k$, there are furthermore the following standard weight structure respectively standard \textbf{t}-structure:
			\begin{itemize}
				\item If $n>0$ and $k\leq 0$, then $A$ is a proper non-positive dg algebra (i.e.~$H^*(A)$ is finite-dimensional and $H^{>0}(A)=0$).
					Hence by \cite[Thm.~A.1]{BruestleYang} (see also \cite[Prop.~5.2.2 and 6.2.1]{Bondarko}) $\perf(A)$ admits the \emph{standard weight structure} $(\cat{T}_{\geq 0},\cat{T}_{\leq 0})$ with $\coheart=\add(A)$.
					The standard weight structure is bounded, so $\cat{T}_{\geq 0}=\Kar\extclos\{A[i]\mid i\leq 0\}$ and $\cat{T}_{\leq 0}=\Kar\extclos\{A[i]\mid i\geq 0\}$.
				\item If $k>0$, then $A$ is a proper positive dg algebra (i.e.~$H^*(A)$ is finite-dimensional, $H^{<0}(A)=0$, and $H^0(A)$ is semisimple).
					By \cite[Thm.~7.1]{KellerNicolas}, $\perf(A)$ admits the \emph{standard \textbf{t}-structure} $(\cat{T}^{\leq 0},\cat{T}^{\geq 0})$ with $\heart=\extclos(A)$.
					The standard \textbf{t}-structure is bounded with finite-length heart, so we have $\cat{T}^{\leq 0}=\extclos\{A[i]\mid i\geq 0\}$ and $\cat{T}^{\geq 0}=\extclos\{A[i]\mid i\leq 0\}$.
				\item If $n=0$, then $A=\kk$ is both a positive and proper non-positive dg algebra.
					Hence $\cat{T}$ admits both the standard weight structure and the standard \textbf{t}-structure, with $\heart=\add(\kk)=\coheart$.
			\end{itemize}
			To avoid confusion with other \textbf{t}-structures we denote the standard \textbf{t}-structure by just $(\cat{T}^{\leq 0},\cat{T}^{\geq 0})$, and write $\heart$ (with no subscript) for its heart.
			Analogously we write $(\cat{T}_{\geq 0},\cat{T}_{\leq 0})$ for the standard weight structure, with coheart $\coheart$.
			The standard weight structure and \textbf{t}-structure can also be constructed directly from $P$ (i.e.~without passing to the dg algebra $A$) by using the correspondence between simple-minded collections and bounded \textbf{t}-structures with finite-length heart, respectively (classical) silting objects and bounded weight structures.
			For details see \cite[\S 5]{KoenigYang} and references therein.
			
			Our main result is the following classification theorem.
			\begin{thm}
				\label{classificationthm}
				Let $\cat{T}$ be the triangulated category generated by a $\proj^n[k]$-object.
				\begin{enumerate}
					\item For $n=0$, any non-trivial \textbf{t}-structure on $\cat{T}$ is obtained by shifting the standard \textbf{t}-structure, and any non-trivial weight structure on $\cat{T}$ is obtained by shifting the standard weight structure.
					\item For $n>0$ and $k>0$, any non-trivial \textbf{t}-structure on $\cat{T}$ is obtained by shifting the standard \textbf{t}-structure.
						$\cat{T}$ does not admit a non-trivial weight structure.
					\item For $n>0$ and $k\leq 0$, any non-trivial weight structure on $\cat{T}$ is obtained by shifting the standard weight structure.
						$\cat{T}$ does not admit a non-trivial \textbf{t}-structure.
				\end{enumerate}
			\end{thm}
			To prove the theorem, we consider each case separately.
			This occupies the remainder of this section.
		\subsection{The degenerate case}
			We first consider the degenerate case of weight structures with $\coheart_w=\{0\}$ and \textbf{t}-structures with $\heart_t=\{0\}$.
			This degenerate case is independent of $n$ and $k$.
			\begin{prop}\leavevmode
				\label{zeroheart}
				\begin{enumerate}
					\item Any weight structure on $\cat{T}$ with zero coheart is trivial.
					\item Any \textbf{t}-structure on $\cat{T}$ with zero heart is trivial.
				\end{enumerate}
			\end{prop}
			\clearpage
			\begin{proof}\leavevmode
				\begin{enumerate}
					\item Let $w=(\cat{T}_{w\geq 0},\cat{T}_{w\leq 0})$ be a weight structure on $\cat{T}$ with $\coheart_w=\{0\}$.
						This easily implies $\cat{T}_{w\geq 0}=\cat{T}_{w\geq i}$ and $\cat{T}_{w\leq 0}=\cat{T}_{w\leq i}$ for all $i\in\ints$, see e.g.~\cite[Lemma~4.1]{HolmJorgensenYang}.
						Let $w_{>0}P\to P\to w_{\leq 0}P\to w_{>0}P[1]$ be a weight decomposition of $P$.
						Since $\cat{T}$ is $nk$-Calabi--Yau, we have
						\begin{equation*}
							\Hom_\cat{T}(w_{\leq 0}P,w_{>0}P[1])\cong\Hom_\cat{T}(w_{>0}P[1],w_{\leq 0}P[nk])^\vee=0
						\end{equation*}
						as $w_{>0}P[1]\in\cat{T}_{w\geq 0}=\cat{T}_{w>0}$ and $w_{\leq 0}P[nk]\in\cat{T}_{w\leq -nk}=\cat{T}_{w\leq 0}$.
						Hence the weight decomposition triangle splits, so $P\cong w_{>0}P\oplus w_{\leq 0}P$.
						But $P$ is indecomposable, so either $P\cong w_{>0}P$ or $P\cong w_{\leq 0}P$, and thus $P\in\cat{T}_{w>0}$ or $P\in\cat{T}_{w\leq 0}$.
						Since $\cat{T}_{w>0}$ and $\cat{T}_{w\leq 0}$ are closed under extensions, positive and negative shifts, and direct summands, it follows that $w$ is trivial.
					\item The proof is analogous to that of 1), using a \textbf{t}-decomposition of $P$.\qedhere
				\end{enumerate}
			\end{proof}
			\begin{rem}
				The same argument shows that any weight decomposition (resp.~$t$-decomposition) of an indecomposable objects in a $d$-Calabi--Yau category is trivial as soon as $\cat{C}_{w>0}\perp\cat{C}_{w\leq 1-d}$ (resp.~$\cat{D}^{t\leq 0}\perp\cat{D}^{t\geq 1-d}$).
				These conditions hold in particular if $w$ (resp.~$t$) has zero coheart (resp.~heart), as seen above.
			\end{rem}
		\subsection{The trivial case \texorpdfstring{$n=0$}{n=0}: exceptional object}
			For $n=0$, the value of $k$ is undefined.
			A $\proj^0$-like object is the same as an exceptional object, and we have
			\begin{equation*}
				\cat{T}\cong\perf(\kk)=\Kb(\rprojfg\kk)=\Db(\rmodfd\kk).
			\end{equation*}
			In this case the classification of \textbf{t}-structures and weight structures is easy and well-known:
			\begin{prop}\leavevmode
				\label{casen0}
				\begin{enumerate}
					\item Any \textbf{t}-structure on $\Db(\rmodfd\kk)$ with non-zero heart is a shift of the standard \textbf{t}-structure.
					\item Any weight structure on $\Kb(\rprojfg\kk)$ with non-zero coheart is a shift of the standard weight structure.
				\end{enumerate}
			\end{prop}
			\begin{proof}
				Any object of $\cat{T}=\Db(\rmodfd\kk)=\Kb(\rprojfg\kk)$ is quasi-isomorphic to its cohomology.
				Therefore the indecomposable objects of $\cat{T}$ are the single-degree complexes $\kk[r]$ for $r\in\ints$.
				Now let $t$ be a non-trivial \textbf{t}-structure, or $w$ a non-trivial weight structure on $\cat{T}$.
				By \cref{zeroheart} there is $0\neq X\in\heart_t$ (resp.~$0\neq X\in\coheart_w$).
				Since $\heart_t$ (resp.~$\coheart_w$) is Karoubi-closed, the indecomposable summands of $X$ also lie in $\heart_t$ (resp.~$\coheart_w$).
				Thus $\kk[r]\in\heart_t$ (resp.~$\kk[r]\in\coheart_w$) for some $r\in\ints$, and therefore $t$ (resp.~$w$) is a shift of the standard \textbf{t}-structure (resp.~the standard weight structure).
			\end{proof}
		\subsection{Non-existence of weight and \textbf{t}-structures}
			\label{nonexistence}
			The following general obstruction to the existence of weight structures respectively \textbf{t}-structures is well-known.
			\begin{lemma}
				\label{cyobstruction}
				Let $\cat{C}$ be a $d$-Calabi--Yau triangulated category.
				\begin{enumerate}
					\item If $d>0$, then any weight structure on $\cat{C}$ has zero coheart.
					\item If $d<0$, then any \textbf{t}-structure on $\cat{C}$ has zero heart.
				\end{enumerate}
			\end{lemma}
			\begin{proof}
				For $X\in\cat{C}$ we have
				\begin{equation*}
					\Hom_\cat{C}(X,X)\cong\Hom_\cat{C}(X,X[d])^\vee.
				\end{equation*}
				If $d>0$ and $X$ lies in a coheart (resp.~if $d<0$ and $X$ lies in a heart), then the right-hand side vanishes and therefore $X=0$.
			\end{proof}
			This immediately rules out non-trivial weight structures (resp.~\textbf{t}-structures) on $\cat{T}$ for $k>0$ (resp.~$k<0$):
			\clearpage
			\begin{cor}
				\label{nostruct}
				Assume $n>0$.
				\begin{enumerate}
					\item For $k>0$, any weight structure on $\cat{T}$ is trivial.
					\item For $k<0$, any \textbf{t}-structure on $\cat{T}$ is trivial.
				\end{enumerate}
			\end{cor}
			\begin{proof}
				$\cat{T}$ is $nk$-Calabi--Yau and $nk\neq 0$, so the claim follows from \cref{cyobstruction} and \cref{zeroheart}.
			\end{proof}
		\subsection{The special case \texorpdfstring{$n>0$, $k=0$}{n>0, k=0}}
			Let $\cat{T}$ be the triangulated category generated by a $\proj^n[0]$-object $P$, with $n>0$.
			In this case $\cat{T}$ is $0$-Calabi--Yau, and admits the standard weight structure $(\cat{T}_{\geq 0},\cat{T}_{\leq 0})$.
			
			We use the model provided by \cref{moritamodel}: we have $\cat{T}\cong\perf(A)=\Kb(\rprojfg A)$, where $A=\kk[x]/(x^{n+1})$.
			In this concrete model, the $\proj^n[0]$-object is $P=A$.
			\begin{lemma}
				\label{pn0homs}
				Let $X\in\cat{T}$ be indecomposable and suppose $X\neq P[i]$ for all $i\in\ints$.
				Then $\Hom_\cat{T}(X,X[>0])\neq 0$ and $\Hom_\cat{T}(X,X[<0])\neq 0$.
			\end{lemma}
			\begin{proof}
				Since $\cat{T}$ is $0$-Calabi--Yau, it suffices to show that $\Hom_\cat{T}(X,X[>0])\neq 0$.
				By shifting, we can without loss of generality write $X$ as a chain comples of finitely generated projective $A$-modules $X=(X^{-r}\to X^{1-r}\to\dots\to X^0)$ with $X^{-r}\neq 0$ and $X^0\neq 0$, where $r\geq 1$ since $X\neq P$.
				Each $X^j$ is a direct sum of copies of $P=A$.
				Without loss of generality we can assume that $X$ is has no contractible summand, which by \cite[Prop.~4.3.18]{KrauseHomologicalTheory} means that each differential $d\: X^j\to X^{j+1}$ is a radical morphism.			
				
				Consider the morphism of chain complexes $f\: X\to X[r]$ for which the (by degree reasons) only non-zero component $f^{-r}\: X^{-r}\to (X[r])^{-r}=X^0$ is given by $f^{-r}=\iota\pi\: X^{-r}\tto P\hookto X^0$, where $\iota$ and $\pi$ denote the inclusion (resp.~projection) of a direct summand $P$ of $X^0$ resp.~$X^{-r}$.
				Then $f$ is not null-homotopic: if $f=d(h)$ for some $h\: X\to X[r-1]$, then in particular $f^{-r}=dh+hd$ would be radical since each differential of $X$ is.
				But we have $\pi f^{-r}\iota=\id_P$, and the identity is not a radical morphism.
			\end{proof}
			\begin{cor}\leavevmode
				\label{casek0}
				\begin{enumerate}
					\item There is no non-trivial \textbf{t}-structure on the triangulated category generated by a $\proj^n[0]$-object.
					\item Up to shift, the only non-trivial weight structure on the triangulated category generated by a $\proj^n[0]$-object is the one defined by the $\proj^n[0]$-object.
				\end{enumerate}
			\end{cor}
			\begin{proof}\leavevmode
				\begin{enumerate}
					\item Let $t=(\cat{T}^{t\leq 0},\cat{T}^{t\geq 0})$ be a \textbf{t}-structure on $\cat{T}$.
						If $\heart_t=\{0\}$, then $t$ is trivial by \cref{zeroheart}.
						If $\heart_t\neq\{0\}$, then by the $0$-Calabi--Yau property we have
						\begin{equation*}
							\Ext_{\heart_t}^1(X,Y)\cong\Hom_\cat{T}(X,Y[1])\cong\Hom_\cat{T}(Y,X[-1])^\vee=0
						\end{equation*}
						for any $X,Y\in\heart_t$, and hence $\heart_t$ is semisimple.
						By \cref{pn0homs} the only indecomposable object in $\heart_t$ can be $P$ (without loss of generality, by shifting $t$).
						But this means $\heart_t=\add(P)\cong\rprojfg\End_\cat{T}(P)\cong\rprojfg \kk[x]/(x^{n+1})$ which does not have all (co)kernels, a contradiction.
					\item Let $w=(\cat{T}_{w\geq 0},\cat{T}_{w\leq 0})$ be a weight structure on $\cat{T}$.
						If $\coheart_w=\{0\}$, then $w$ is trivial by \cref{zeroheart}.
						If $\coheart_w\neq\{0\}$, then by \cref{pn0homs} the only indecomposable objects in $\coheart_w$ can be $P$ (without loss of generality, by shifting).
						But then $\cat{T}_{\leq 0}=\Kar\extclos\{P[i]\mid i\geq 0\}\subseteq\cat{T}_{w\leq 0}$ and $\cat{T}_{\geq 0}=\Kar\extclos\{P[i]\mid i\leq 0\}\subseteq\cat{T}_{w\geq 0}$, so $w$ must be the standard weight structure.\qedhere
				\end{enumerate}
			\end{proof}
		\subsection{The case \texorpdfstring{$n>0$, $k>0$}{n>0, k>0}}
			If $n>0$ and $k>0$, then $\cat{T}$ admits the standard \textbf{t}-structure $(\cat{T}^{\leq 0},\cat{T}^{\geq 0})$ defined by $\heart=\extclos_\cat{T}(P)$.
			This is bounded with finite-length heart, and $P$ is the unique simple object of $\heart$.
			
			The indecomposable objects in the standard heart can be listed explicitly.
			Surprisingly, the case $n=1$ (i.e.~when $P$ is spherical) considered in \cite{HolmJorgensenYang} is the most complicated.
			\begin{lemma}\leavevmode
				\label{indecinstdheart}
				\begin{enumerate}
					\item If $k>1$, then $\heart=\add(P)$.
					\item If $k=1$ and $n>1$, then $\heart=\add(\{P,X_2\})$ where $X_2$ is defined by the unique (up to isomorphism) non-split triangle $P\to X_2\to P\longto{x}P[1]$.
					\item If $k=1$ and $n=1$, then $\heart=\add(\{X_i\mid i\geq 1\})$, where $X_i$ is uniserial with $i$ simple subquotients $P$.
				\end{enumerate}
			\end{lemma}
			\begin{proof}\leavevmode
				\begin{enumerate}
					\item For $k>1$ we have $\Ext_\heart^1(P,P)=\Hom_\cat{T}(P,P[1])=0$, and hence $\heart=\add(P)$.
					\item For $k=1$ we have $\Ext_\heart^1(P,P)=\Hom_\cat{T}(P,P[1])\cong\kk$, so there is a unique (up to isomorphism) non-split triangle
						\begin{equation*}
							P\to X_2\to P\longto{x}P[1].
						\end{equation*}
						By applying $\Hom_\cat{T}(P,-)$ to this triangle we obtain the long exact sequence
						\begin{equation*}
							\Hom_\cat{T}(P,P[i])\to\Hom_\cat{T}(P,X_2[i])\to \Hom_\cat{T}(P,P[i])\to\Hom_\cat{T}(P,P[i+1]),
						\end{equation*}
						where the connecting homomorphism $\Hom_\cat{T}(P,P[i])\to\Hom_\cat{T}(P,P[i+1])$ is given by postcomposition with $x\: P\to P[1]$.
						Since $P$ is a $\proj^n[1]$-object, we have $\Hom_\cat{T}(P,P[i])=\kk$ for $0\leq i\leq n$ and $\Hom_\cat{T}(P,P[i])=0$ otherwise, and moreover the connecting homomorphism is an isomorphism for $0\leq i\leq n-1$.
						From this (and the analogous argument using the long exact sequence obtained by applying $\Hom_\cat{T}(-,P)$) it follows that
						\begin{align*}
							\Hom_\cat{T}(P,X_2[i])&\cong\begin{cases}
								\kk&\text{if $i=0$ or $i=n$},\\
								0&\text{otherwise},
							\end{cases}\\
							\Hom_\cat{T}(X_2,P[i])&\cong\begin{cases}
								\kk&\text{if $i=0$ or $i=n$},\\
								0&\text{otherwise}.
							\end{cases}
						\end{align*}
						Hence for $n>1$, there is no non-split extension between $X_2$ and the unique simple object $P$, and thus $\heart=\add(\{P,X_2\})$.
					\item For $k=1$ and $n=1$, the objects $X_i$ for $i\geq 1$ are constructed inductively via unique non-split triangles $P\to X_{i+1}\to X_i\to P[1]$ and $X_i\to X_{i+1}\to P\to X_i[1]$, starting from $X_1=P$.
						Suppose that $X_i$ has been constructed.
						As in the proof of 2), the long exact sequences obtained by applying $\Hom_\cat{T}(P,-)$ and $\Hom_\cat{T}(-,P)$ show that $\Hom_\cat{T}(P,X_i[1])\cong\kk$ and $\Hom_\cat{T}(X_i,P[1])\cong\kk$.
						That the unique non-split triangles resulting from these extensions define the same object $X_{i+1}$ can be seen via the octahedral axiom.\qedhere
				\end{enumerate}
			\end{proof}
			\begin{lemma}
				\label{neghomforposk}
				Let $X\in\cat{T}$ with $X\in\cat{T}^{\geq 0}\cap\cat{T}^{\leq r}$ and $X\notin\cat{T}^{>0}$, $X\notin\cat{T}^{<r}$.
				If $r>0$, then $\Hom_\cat{T}(X,X[-r])\neq 0$.
			\end{lemma}
			\begin{proof}
				Since $X\in\cat{T}^{\geq 0}\cap\cat{T}^{\leq r}$, we have
				\begin{equation*}
					\Hom_\cat{T}(X,X[-r])\cong\Hom_\cat{T}(X,\tau_{\leq r}(X[-r]))\cong\Hom_\cat{T}(\tau_{\geq r}X,\tau_{\leq r}(X[-r]))
				\end{equation*}
				where $\tau_{\leq r}$ and $\tau_{\geq r}$ are the truncation functors for the standard \textbf{t}-structure.
				Now both $\tau_{\geq r}X$ and $\tau_{\leq r}(X[-r])$ are non-zero and lie in $\heart[-r]$.
				Since $\heart$ is finite-length with the unique simple object $P$, it follows that $\Hom_\cat{T}(\tau_{\geq r}X,\tau_{\leq r}(X[-r]))\neq 0$.
			\end{proof}
			\begin{prop}
				\label{casekpos}
				If $n>0$ and $k>0$, then the only non-trivial \textbf{t}-structures on $\cat{T}$ are the shifts of the standard \textbf{t}-structure.
			\end{prop}
			\begin{proof}
				Let $t$ be a \textbf{t}-structure on $\cat{T}$.
				If $\heart_t=\{0\}$ then $t$ is trivial by \cref{zeroheart}, so assume that there is a non-zero indecomposable object $X\in\heart_t$.
				By shifting $X$ we can without loss of generality assume that with respect to the standard \textbf{t}-structure we have $X\in\cat{T}^{\geq 0}\cap\cat{T}^{\leq r}$ and $X\notin\cat{T}^{>0}$, $X\notin\cat{T}^{<r}$ for some $r\geq 0$.
				If $r>0$, then by \cref{neghomforposk} we have $\Hom_\cat{T}(X,X[-r])\neq 0$ and $-r<0$, a contradiction to $X$ being in the heart of a \textbf{t}-structure.
				
				If $r=0$ then $X$ is also in the heart of the standard \textbf{t}-structure.
				By \cref{indecinstdheart}, if $k>1$ then $X=P$, and thus $t$ is the standard \textbf{t}-structure.
				For $k=1$, $X$ is a uniserial self-extension of copies of $P$.
				Since short exact sequences in the heart of a \textbf{t}-structure are precisely the triangles in the ambient triangulated category, it follows that (co)kernels in $\heart_t$ agree with the (co)kernels in the heart $\heart$ of the standard \textbf{t}-structure.
				But in the standard heart, $P$ can be obtained as the kernel of the endomorphism of $X$ sending the copy of $X$ at the top to the copy one layer below.
				Hence if $X\in\heart_t$, then also $P\in\heart_t$, and therefore $t$ is the standard \textbf{t}-structure.
			\end{proof}
		\subsection{The case \texorpdfstring{$n>0$, $k<0$}{n>0, k<0}}
			The basic proof idea for $n>0$, $k<0$ is similar to the argument for $n>0$, $k=0$ in \cref{pn0homs}.
			However we cannot directly apply \cite[Prop.~4.3.18]{KrauseHomologicalTheory} in this setting.
			The following lemma serves as a replacement, and its proof is similar to that of \cite[Prop.~4.3.18]{KrauseHomologicalTheory}.
			
			By \cref{moritamodel} we have $\cat{T}=\perf(A)$ for $A=\kk[x]/(x^{n+1})$ with $|x|=k<0$.
			As a graded $A$-module, any $X\in\perf(A)$ can be written as $X=\bigoplus_{r\in\ints} P[-r]^{\oplus a_r}$ with only finitely many $a_r\neq 0$.
			The differential $d_X\: X\to X[1]$ is a morphism of graded $A$-modules, which we will write in matrix form with respect to such a chosen decomposition of $X$.
			For instance, we may consider $X=P\oplus P[-1]$ with differential $d=\begin{smallpmatrix}0&0\\1&0\end{smallpmatrix}\: X=P\oplus P[-1]\to P[1]\oplus P=X[1]$.
			\begin{lemma}
				\label{dgradicallemma}
				If $X$ is a perfect dg $A$-module without a contractible summand, then $d(X)\subseteq xX$.
			\end{lemma}
			\begin{proof}
				Suppose that there is $v_0\in X$ such that $d(v_0)\notin xX$.
				For easier notation, by shifting we may assume without loss of generality that $|v_0|=0$.
				As graded $A$-modules, $v_0$ generates a summand $P$ of $X$, and $d(v_0)$ generates a summand $P[-1]$.
				Thus we may write $X=P\oplus P[-1]\oplus\bar{X}$ where $\bar{X}$ is a direct sum of copies of shifts of $P$.
				By construction the differential is then of the form $d=\begin{smallpmatrix}0&0&*\\1&0&*\\0&0&*\end{smallpmatrix}$.
				
				Choose a decomposition $\bar{X}=\bigoplus_{i=1}^m P[-r_i]$ and write $v_i$ for a generator of the $i$-th summand.
				Our goal is to perform a base change so that in this basis the right column of $d$ is zero except for the bottom entry.
				For this, write $d(v_i)=a_id(v_0)+*$ for some (homogeneous) $a_i\in A$, where $*$ denotes the components of $d(v_i)$ that are not in the fixed $P[-1]$.
				Then $v_i-a_iv_0$ also generates a summand $P[-r_i]$ of $\bar{X}$, and the $d(v_0)$-coefficient of $d(v_i-a_iv_0)$ vanishes (by construction).
				In terms of these generators of summands of $\bar{X}$, $d$ is thus of the form $\begin{smallpmatrix}0&0&*\\1&0&0\\0&0&*\end{smallpmatrix}$.
				Since $d^2=0$, the top right entry must moreover be $0$.
				Hence we have found a summand $(P\oplus P[-1],d=\begin{smallpmatrix}0&0\\1&0\end{smallpmatrix})$ of $X$, and this summand is contractible since $h=\begin{smallpmatrix}0&1\\0&0\end{smallpmatrix}\: P\oplus P[-1]\to P[-1]\oplus P[-2]$ is a null-homotopy for the identity.
			\end{proof}
			Recall from \cref{standardstructs} that $\cat{T}$ admits the standard weight structure with $\coheart=\add(P)$.
			We write $\omega_{>0}$ and $\omega_{\leq 0}$ for weight decompositions with respect to this.
			\begin{prop}
				\label{poshomfornegk}
				Let $X\in\cat{T}_{\geq r}\cap\cat{T}_{\leq 0}$ for some $r\leq 0$, such that $X\notin\cat{T}_{>r}$ and $X\notin\cat{T}_{<0}$.
				Then $\Hom_\cat{T}(X,X[-r])\neq 0$.
			\end{prop}
			\begin{proof}
				Without loss of generality we may assume that $X$ has no contractible summand.
				We can write $X=\bigoplus_{r\leq s\leq 0} P[-s]^{\oplus a_s}$ as graded $A$-module.
				Observe that the projection $\pi\: X\to P[-r]^{\oplus a_r}$ and the inclusion $\iota\: P^{\oplus a_0}\to X$ are dg morphisms, providing weight decompositions $X\to\omega_{\leq r}X=P[-r]^{\oplus a_r}$ and (by shifting) $X[-r]\to\omega_{\geq r}(X[-r])=P[-r]^{\oplus a_0}$, respectively.
				Define $f\: X\to X[-r]$ as the composition
				\begin{equation*}
					X\longto{\pi} \omega_{\leq r}X=P[-r]^{\oplus a_r}\longto{\tilde{f}}P[-r]^{\oplus a_0}=\omega_{\geq r}X[-r]\longto{\iota} X[-r]
				\end{equation*}
				where one component of $\tilde{f}$ is the identity and all other components are $0$.
				Then $f$ is a dg morphism (as it is a composition of dg morphisms).
				Explicitly, as a morphism of graded $A$-modules $f$ identifies a summand $P[-r]$ of $X$ with a summand $P[-r]$ of $X[-r]$ (and vanishes on all other summands).
				
				$f$ cannot be null-homotopic: by \cref{dgradicallemma} we have $d(X)\subseteq xX$, which means that if if $f=dh+hd$ for some $h\: X\to X[-1]$ then the image of $f=dh+hd$ must lie in $xX$, a contradiction.
			\end{proof}
			\begin{cor}
				\label{casekneg}
				For $n>0$ and $k<0$, any non-trivial weight structure on $\cat{T}$ is a shift of the standard weight structure.
			\end{cor}
			\begin{proof}
				Let $w$ be a non-trivial weight structure on $\cat{T}$.
				By \cref{zeroheart} we have $\coheart_w\neq\{0\}$, and by \cref{poshomfornegk} the only indecomposable object in $\coheart_w$ can be $P[i]$ for some $i\in\ints$.
				Hence $w$ is a shift of the standard weight structure.
			\end{proof}
		\subsection{Proof of the classification}
			The classification result \cref{classificationthm} now follows by collecting the partial results from above.
			\begin{proof}[Proof of \cref{classificationthm}]
				The standard weight structure and \textbf{t}-structure are explained in \cref{standardstructs}.
				The case $n=0$ is \cref{casen0}, and the case $n>0$, $k=0$ is \cref{casek0}.
				The case $n>0$, $k>0$ is \cref{casekpos} for \textbf{t}-structures and \cref{nostruct} for weight structures, and the case $n>0$, $k<0$ is \cref{casekneg} for weight structures and \cref{nostruct} for \textbf{t}-structures.
			\end{proof}
	\section{The Koszul-dual model}
		\label{koszulduality}
		A different model of $\cat{T}$ can be obtained via dg Koszul duality \cite[\S 10]{KellerDerivingDGCats}.
		For an augmented dg algebra $A$, its \emph{dg Koszul dual} $A^{!,\dg}$ is defined as
		\begin{equation*}
			A^{!,\dg}=\RHom_A(\kk,\kk).
		\end{equation*}
		This is a special case of the general definition of the dg Koszul dual of an augmented dg category from \cite[\S 10.2]{KellerDerivingDGCats}.
		If $A$ is proper non-positive, then $A^{!,\dg}$ is positive by \cite[Thm.~A.1]{BruestleYang}, and if $A$ is proper positve then $A^{!,\dg}$ is non-positive by \cite[Lemma~5.2]{KellerNicolas}, see also \cite[Ex.~4.4]{WeightTPaper}.
		
		In this section we explicitly compute the dg Koszul dual of $A=\kk[x]/(x^{n+1})$, and describe the equivalence that dg Koszul duality provides.
		For this, we consider three cases separately: $n>1$, $n=1$ with $k\neq 1$, and $n=1$ with $k=1$.
		The case $n=0$ is trivial: in this case $A^{!,\dg}=\kk$, and the Koszul duality equivalence reduces to the equivalence $\Kb(\rprojfg \kk)=\perf(\kk)\cong\Dfd(\kk)=\Db(\rmodfd\kk)$.
		\subsection{dg Koszul duals of graded algebras}
			We first discuss general results about the dg Koszul dual of graded algebras.
			For this we have to set up some notation.
			
			Let $A$ be a graded algebra.
			We write $\grrMod A$ for the category of graded right $A$-modules, and $\K(\grrMod A)$ for its homotopy category.
			The dg category of chain complexes over $\grrMod A$ is denoted by $\C_\dg(\grrMod A)$.
			The shift with respect to the grading of $A$ is denoted by $\<1\>$ and the shift of chain complexes by $[1]$.
			
			We can also view $A$ as a dg algebra with trivial differential.
			The dg category of right dg $A$-modules is denoted by $\C_\dg(A)$.
			Explicitly, for two dg $A$-modules $X$, $Y$, we have
			\begin{equation*}
				\Hom_{\C_\dg(A)}(X,Y)=\coprod_{i\in\ints} \Hom_{\grrMod A}(X,Y\<i\>)
			\end{equation*}
			with the induced differential $d(f)=df-(-1)^{|f|}fd$.
			The homotopy category and derived category of $A$ are denoted by $\K(A)$ and $\D(A)$, respectively.
			
			By the definition from \cite[\S 10.2]{KellerDerivingDGCats} (see \cite[\S 4.1]{WeightTPaper} for details), the dg Koszul dual of an augmented dg algebra $A$ is $A^{!,\dg}=\RHom_A(\kk,\kk)=\Hom_{\C_\dg(A)}(P_\kk,P_\kk)$, where $P_\kk$ is a K-projective resolution of $\kk$, i.e.~a dg $A$-module that is quasi-isomorphic to $\kk$ and satisfies $\Hom_{\K(A)}(P_\kk,M)=0$ for all acyclic dg $A$-modules $M$.
			If $A$ has trivial differential, K-projective resolutions can be obtained as (coproduct) total complexes of graded projective resolutions:
			\begin{lemma}
				\label{kprojres}
				Let $A$ be a graded algebra and $P$ a bounded above complex of graded projective $A$-modules.
				\begin{enumerate}
					\item $\Tot^\amalg(P)$ is a K-projective dg $A$-module.
					\item If $\pi\: P\to X$ is a projective resolution of a graded $A$-module $X$, then $\Tot^\amalg(\pi)\: \Tot^\amalg(P)\to\Tot^\amalg(X)$ is a K-projective resolution of $\Tot^\amalg(X)=X$ as a dg $A$-module.
				\end{enumerate}
			\end{lemma}
			\begin{proof}\leavevmode
				\begin{enumerate}
					\item The category of K-projective dg $A$-modules contains $A$ and is closed under cones, shifts and arbitrary coproducts in $\K(A)$, and thus also under homotopy colimits.
						The triangulated functor $\Tot^\amalg\: \K(\grrMod A)\to\K(A)$ preserves arbitrary coproducts, and hence commutes with homotopy colimits.
						We obviously have $\Tot^\amalg(A)=A$, and thus $\Tot^\amalg(P)$ is K-projective if $P$ is bounded.
						The claim follows since every cochain complex is isomorphic to the homotopy colimit over its brutal truncations, see e.g.~\cite[Ex.~4.2.2]{KrauseHomologicalTheory}, which are bounded if $P$ is bounded above.
					\item This is immediate from 1) since $\Tot^\amalg$ preserves quasi-isomorphisms.\qedhere
				\end{enumerate}
			\end{proof}
			Hence the dg Koszul dual of an augmented graded algebra $A$ is given by $A^{!,\dg}=\End_{\C_\dg(A)}(\Tot^\amalg(P_\kk))$, where $P_\kk$ is a projective resolution of $\kk$ as graded $A$-module.
			As a graded $A$-module, $\Tot^\amalg(P_\kk)$ is the (potentially infinite) coproduct $\coprod_{i\in\ints} P_\kk^i\<-i\>$.
			This can make $A^{!,\dg}$ difficult to compute directly.
			
			However, following the approach in \cite[\S 5]{KalckYangGentle} we can relate $A^{!,\dg}$ to the $\ints$-graded dg algebra
			\begin{equation*}
				\mathcal{E}=\End_{\C_\dg(\grrMod A)}^*(P_\kk)=\coprod_{j\in\ints} \Hom_{\C_\dg(\grrMod A)}(P_\kk,P_\kk\<j\>).
			\end{equation*}
			Here by a $\ints$-graded dg algebra we mean a $\ints\times\ints$-graded algebra with a differential of bidegree $(1,0)$ satisfying the Leibniz rule $d(ab)=d(a)b+(-1)^{|a|_1}ad(b)$ for $a$ homogeneous of bidegree $|a|=(|a|_1,|a|_2)$.
			By definition, the cohomology of $\mathcal{E}$ is the $\ints\times\ints$-graded algebra $\Ext_{\grrMod A}^{*,*}(\kk,\kk)=\coprod_{i,j\in\ints} \Ext_{\grrMod A}^i(\kk,\kk\<j\>)$.
			
			The following lemma is a variant of \cite[Lemma~5.3]{KalckYangGentle} for diagonal-finite unbounded complexes, and the proof is very similar.
			A similar argument (again in the bounded case) also appears in the proof of \cite[Thm.~39]{Schnuerer}.
			Here we call a complex $X$ of $\ints$-graded $A$-modules \emph{diagonal-finite} if for each $r\in\ints$ there are only finitely many $i\in\ints$ with $X^{i,r-i}\neq 0$.
			Note that $X$ is diagonal-finite if and only if $\Tot^\Pi(X)=\Tot^\amalg(X)$.
			\begin{lemma}
				\label{dgvsgradedext}
				Let $A$ be a graded algebra, $X$ any complex of graded $A$-modules and $Y$ a diagonal-finite complex of graded $A$-modules.
				Then there is a natural isomorphism of dg vector spaces
				\begin{equation*}
					\Hom_{\C_\dg(A)}(\Tot^\amalg(X),\Tot^\amalg(Y))\cong\Tot^\Pi(\Hom_{\C_\dg(\grrMod A)}^*(X,Y)).
				\end{equation*}
				In particular there is an isomorphism of dg algebras
				\begin{equation*}
					\End_{\C_\dg(A)}(\Tot^\amalg(Y))\cong\Tot^\Pi(\End_{\C_\dg(\grrMod A)}^*(Y)).
				\end{equation*}
			\end{lemma}
			\begin{proof}
				Observe that if $Y$ is diagonal-finite, then $\coprod_{p\in\ints} Y[p]\<-p\>\cong\prod_{p\in\ints} Y[p]\<-p\>$ as chain complexes of graded $A$-modules.
				Indeed, in bidegree $(i,j)$ we have
				\begin{equation*}
					\begin{split}
						\biggl(\coprod_{p\in\ints} Y[p]\<-p\>\biggr)^{i,j}&=\coprod_{p\in\ints} (Y[p]\<-p\>)^{i,j}=\coprod_{p\in\ints} Y^{i+p,j-p}\\
						&=\prod_{p\in\ints} Y^{i+p,j-p}=\prod_{p\in\ints} (Y[p]\<-p\>)^{i,j}=\biggl(\prod_{p\in\ints} Y[p]\<-p\>\biggr)^{i,j}
					\end{split}
				\end{equation*}
				since by assumption $Y^{i+p,j-p}$ is non-zero for only finitely many $p$.
				
				By \cite[Thm.~5.1]{KalckYangGentle} and the above we thus have natural isomorphisms
				\begin{equation*}
					\begin{split}
						\Hom_{\C_\dg(A)}(\Tot^\amalg(X),\Tot^\amalg(Y))&\cong\Hom_{\C_\dg(\grrMod A)}\biggl(X,\coprod_{p\in\ints} Y[p]\<-p\>\biggr)\\
						&\cong\Hom_{\C_\dg(\grrMod A)}\biggl(X,\prod_{p\in\ints} Y[p]\<-p\>\biggr)\\
						&\cong\prod_{p\in\ints} \Hom_{\C_\dg(\grrMod A)}(X,Y[p]\<-p\>).
					\end{split}
				\end{equation*}
				Similarly to the proof of \cite[Lemma~5.3]{KalckYangGentle} we see that this dg vector space is isomorphic to $\Tot^\Pi(\Hom_{\C_\dg(\grrMod A)}^*(X,Y))$.
				Indeed, in degree~$i$ we have
				\begin{equation*}
					\begin{split}
						\biggl(\prod_{p\in\ints} \Hom_{\C_\dg(\grrMod A)}(X,Y[p]\<-p\>)\biggr)^i&=\prod_{p\in\ints} \Hom_{\C_\dg(\grrMod A)}(X,Y[p+i]\<-p\>)\\
						&=\Tot^\Pi(\Hom_{\C_\dg(\grrMod A)}^*(X,Y))^i,
					\end{split}
				\end{equation*}
				and this agrees with the differential.
			\end{proof}
			\begin{cor}
				\label{dgvsgradedkoszul}
				Let $A$ be an augmented graded algebra such that the projective resolution $P_\kk$ of $\kk$ as graded $A$-module is diagonal-finite.
				Then
				\begin{equation*}
					A^{!,\dg}\cong\Tot^\Pi(\End_{\C_\dg(\grrMod A)}^*(P_\kk)).
				\end{equation*}
			\end{cor}
			\begin{proof}
				This is immediate from \cref{kprojres,dgvsgradedext}.
			\end{proof}
			\begin{rem}
				If $A$ is a Koszul algebra, then the projective resolution of $\kk$ is diagonal-finite if and only if it is finite, i.e.~if and only if $\gldim(A)<\infty$.
				In this case, \cite[Thm.~39]{Schnuerer} shows that $H^0(A^{!,\dg})\cong A^!$ is the classical Koszul dual.
				Moreover, all other cohomologies of $A^{!,\dg}$ vanish: since $A$ is a positive dg algebra, its dg Koszul dual $A^{!,\dg}$ is a non-positive dg algebra, i.e.~$H^i(A^{!,\dg})=0$ for $i>0$.
				By Koszulity, $\Tot^\amalg(P_\kk)$ is generated in degree~$0$ as an $A$-module, and it is concentrated in non-negative degrees.
				Hence there are no $A$-linear maps $\Tot^\amalg(P_\kk)\to\Tot^\amalg(P_\kk)[i]$ for $i<0$, and thus also $H^i(A^{!,\dg})=0$ for $i<0$.
			\end{rem}
		\subsection{Computing the Koszul dual}
			Now let $A=\kk[x]/(x^{n+1})$ with $|x|=k$.
			Our goal is to determine $A^{!,\dg}$ explicitly via \cref{dgvsgradedkoszul}.
			As mentioned above, the case $n=0$ is not very interesting, so we assume $n\geq 1$ from now on.
			\begin{lemma}
				\label{projresdiagfin}
				The graded projective resolution of $\kk$ over $A=\kk[x]/(x^{n+1})$ with $|x|=k$ is diagonal-finite unless $n=1$ and $k=1$.
			\end{lemma}
			\begin{proof}
				The graded projective resolution $P_\kk$ of $\kk$ is given by
				\begin{equation}
					\label{projres}
					\dots\longto{x} A\<-(2n+2)k\>\longto{x^n} A\<-(n+2)k\>\longto{x}A\<-(n+1)k\>\longto{x^n}A\<-k\>\longto{x}A,
				\end{equation}
				and that this is diagonal-finite unless $n=k=1$ is an easy degree counting argument.
			\end{proof}
			The cohomology of the $\ints$-graded dg algebra $\End_{\C_\dg(\grrMod A)}^*(P_\kk)$ is easy to compute:
			\begin{lemma}
				\label{extalgebra}
				We have
				\begin{equation*}
					\Ext_{\grrMod A}^{*,*}(\kk,\kk)\cong\begin{cases}
						\kk[s]&\text{if }n=1,\\
						\kk[s,t]/(s^2)&\text{if }n>1,
					\end{cases}
				\end{equation*}
				where $|s|=(1,-k)$ and $|t|=(2,-(n+1)k)$.
			\end{lemma}
			\begin{proof}
				Applying $\Hom_{\grrMod A}(-,\kk\<j\>)$ to the graded projective resolution $P_\kk$ of $\kk$ from \cref{projres} yields the claimed dimensions of $\Ext_{\grrMod A}^{i,j}(\kk,\kk)$.
				To determine the multiplicative structure we write down chain maps representing the ext classes.
				Explicitly, $s$ is represented by
				\begin{equation*}
					\begin{tikzcd}
						\dots\arrow{r}{x^n}
						&A\<-(n+2)k\>\arrow{r}{x}\arrow{d}{1}
						&A\<-(n+1)k\>\arrow{r}{x^n}\arrow{d}{x^{n-1}}
						&A\<-k\>\arrow{r}{x}\arrow{d}{1}
						&A\arrow{d}\\
						\dots\arrow{r}{x}
						&A\<-(n+2)k\>\arrow{r}{x^n}
						&A\<-2k\>\arrow{r}{x}
						&A\<-k\>\arrow{r}
						&0
					\end{tikzcd}
				\end{equation*}
				and $t$ is represented by
				\begin{equation*}
					\begin{tikzcd}
						\dots\arrow{r}{x^n}
						&A\<-(n+2)k\>\arrow{r}{x}\arrow{d}{1}
						&A\<-(n+1)k\>\arrow{r}{x^n}\arrow{d}{1}
						&A\<-k\>\arrow{r}{x}\arrow{d}
						&A\arrow{d}\\
						\dots\arrow{r}{x^n}
						&A\<-(n+2)k\>\arrow{r}{x}
						&A\<-(n+1)k\>\arrow{r}
						&0\arrow{r}
						&0\rlap{.}
					\end{tikzcd}
				\end{equation*}
				It is easy to see that $s$ and $t$ commute, and that $t^i$ and $st^i$ are not null-homotopic for any $i\geq 0$.
				Moreover, for $n=1$ we have $s^2=t$, while for $n>1$ a null-homotopy for $s^2$ is given by
				\begin{equation*}
					\begin{tikzcd}[baseline=(\tikzcdmatrixname-\the\pgfmatrixcurrentrow-1.base)]
						\dots\arrow{r}{x^n}
						&A\<-(n+2)k\>\arrow{r}{x}\arrow{d}{x^{n-1}}\arrow[dashed]{ld}[swap]{0}
						&A\<-(n+1)k\>\arrow{r}{x^n}\arrow{d}{x^{n-1}}\arrow[dashed]{ld}[swap]{x^{n-2}}
						&A\<-k\>\arrow{r}{x}\arrow{d}\arrow[dashed]{ld}[swap]{0}
						&A\arrow{d}\arrow[dashed]{ld}\\
						\dots\arrow{r}{x^n}
						&A\<-3k\>\arrow{r}{x}
						&A\<-2k\>\arrow{r}
						&0\arrow{r}
						&0\rlap{.}
					\end{tikzcd}\qedhere
				\end{equation*}
			\end{proof}
			\subsubsection{The case \texorpdfstring{$n>1$}{n>1}}
				For $n>1$, we define the $\ints$-graded dg algebra
				\begin{equation*}
					E=\kk\<s,h_0,\dots,h_{n-2}\>/(h_ih_j,s^2h_i-h_is^2, h_{i+1}sh_j-h_ish_{j+1}, s^2h_i-h_{n-2}sh_{i+1})
				\end{equation*}
				with the bigrading given by $|s|=(1,-k)$ and $|h_i|=(1,(i-n)k)$, and the differential defined on generators by $d(s)=0$, $d(h_{n-2})=s^2$ and $d(h_i)=sh_{i+1}+h_{i+1}s$ for $i<n-2$.
				The differential is extended to all of $E$ by the Leibniz rule.
				$E$ is a well-defined $\ints$-graded dg algebra since the relations of $E$ are stable under $d$ and $d^2(h_i)=0$.
				
				It is easy to see that a Gr\"obner basis of the ideal is given by
				\begin{align*}
					\{s^2h_{n-2}-h_{n-2}s^2\}&\cup\{h_ih_j\mid 0\leq i,j\leq n-2\}\\
					&\cup\{s^2h_i-h_{n-2}sh_{i+1}, h_is^2-h_{n-2}sh_{i+1}\mid 0\leq i<n-2\}\\
					&\cup\{h_ish_{j+1}-h_{i+j+1}sh_0\mid 0\leq i,j\leq n-2, i+j<n-2\}\\
					&\cup\{h_ish_{j+1}-h_{n-2}sh_{j-(n-2-i)+1}\mid 0\leq i,j\leq n-2, i+j\geq n-2\},
				\end{align*}
				and from this it follows that a basis of $E$ is given by the monomials
				\begin{equation*}
					h_{n-2}s\dots sh_{n-2}sh_ish_0s\dots sh_0s^j,\qquad sh_{n-2}s\dots sh_{n-2}sh_ish_0s\dots sh_0s^j
				\end{equation*}
				for $j\geq 0$ and $0\leq i<n-2$.
				To work with this monomial basis of $E$, it is convenient to set
				\begin{equation*}
					v_{r,l}=\underbrace{h_{n-2}s\dots sh_{n-2}}_{\lfloor\frac{r}{n-2}\rfloor\text{ many }h_{n-2}}sh_is\underbrace{h_0sh_0s\dots sh_0sh_0}_{l-1-\lfloor\frac{r}{n-2}\rfloor\text{ many }h_0}
				\end{equation*}
				for $l\geq 1$ and $0\leq r\leq l(n-2)$, with $i=r\bmod (n-2)$.
				\begin{lemma}
					\label{differentialformulas}
					Let $l\geq 1$ and $0\leq r\leq l(n-2)$.
					\begin{enumerate}
						\item If $r<l(n-2)$, then $s^2v_{r,l}=v_{r,l}s^2=v_{r+n-1,l+1}$.
						\item The differential of $E$ is given by
							\begin{align*}
								d(v_{r,l})&=\begin{cases}
									sv_{r+1,l}+v_{r+1,l}s&\text{if }r<l(n-2),\\
									sv_{r-(n-2),l-1}s^2+v_{r-(n-2),l-1}s^3&\text{if }r=l(n-2), l>1,\\
									s^2&\text{if }r=n-2, l=1,
								\end{cases}\\
								d(sv_{r,l})&=\begin{cases}
									-v_{r+n,l+1}-sv_{r+1,l}s&\text{if }r<l(n-2)-1,\\
									-v_{r+1,l}s^2-sv_{r+1,l}s&\text{if }r=l(n-2)-1,\\
									-v_{r-(n-2),l-1}s^4-sv_{r-(n-2),l-1}s^3&\text{if }r=l(n-2), l>1,\\
									-s^3&\text{if }r=n-2, l=1,\\
								\end{cases}\\
								d(v_{r,l}s)&=\begin{cases}
									sv_{r+1,l}s+v_{r+n,l+1}&\text{if }r<l(n-2)-1,\\
									sv_{r+1,l}s+v_{r+1,l}s^2&\text{if }r=l(n-2)-1,\\
									sv_{r-(n-2),l-1}s^3+v_{r-(n-2),l-1}s^4&\text{if }r=l(n-2), l>1,\\
									s^3&\text{if }r=n-2, l=1,
								\end{cases}\\
								d(sv_{r,l}s)&=\begin{cases}
									-v_{r+n,l+1}s-sv_{r+n,l+1}&\text{if }r<l(n-2)-1,\\
									-v_{r+1,l}s^3-sv_{r+1,l}s^2&\text{if }r=l(n-2)-1,\\
									-v_{r-(n-2),l-1}s^5-sv_{r-(n-2),l-1}s^4&\text{if }r=l(n-2), l>1,\\
									-s^4&\text{if }r=n-2, l=1.
								\end{cases}
							\end{align*}
					\end{enumerate}
				\end{lemma}
				\begin{proof}\leavevmode
					\begin{enumerate}
						\item This is immediate from the relations of $E$.
						\item $d(v_{r,l})$ is easily computed using the Leibniz rule and the relations of $E$.
							The remaining formulas follow from this by the Leibniz rule and 1).\qedhere
					\end{enumerate}
				\end{proof}
				\begin{prop}\leavevmode
					\label{dgalgcohomology}
					\begin{enumerate}
						\item The cohomology of $E$ is $H^*(E)\cong\kk[s,t]/(s^2)$, where $t$ is represented by $sh_0+h_0s$.
						\item $E$ is not formal.
					\end{enumerate}
				\end{prop}
				\begin{proof}\leavevmode
					\begin{enumerate}
						\item A straightforward but tedious computation using the explicit formulas for the differential from \cref{differentialformulas} shows that $H^{2l}(E)$ is spanned by $sv_{0,l}+v_{0,l}s$ for $l\geq 1$, and $H^{2l-1}(E)$ is spanned by $v_{n-1,l}+sv_{0,l-1}s$ for $l\geq 1$.
							Moreover $H^0(E)$ and $H^1(E)$ are spanned by $1$ and $s$, respectively.
							Setting $t=h_0s+sh_0$, one sees from the relations of $E$ that $st=s(sh_0+h_0s)=(sh_0+h_0s)s=ts$, and moreover $t^r=(sh_0+h_0s)^r=sv_{0,r}+v_{0,r}s$ and $st^r=s^2v_{0,r}+sv_{0,r}s=v_{n-1,r+1}+sv_{0,r}s$.
							Hence $H^*(E)\cong\kk[s,t]/(s^2)$, as claimed.
						\item Recall (see e.g.~\cite[\S 2.4]{BabenkoTaimanov}) that for cohomology classes $a$, $b$, $c$ with $ab=0$ and $bc=0$, the triple Massey product $\<a,b,c\>$ consists of all cohomology classes represented by $av+(-1)^{|a|+1}uc$ for some $u$, $v$ with $d(u)=ab$, $d(v)=bc$.
							It \emph{vanishes} if $0=av+(-1)^{|a|+1}uc$ for some $u$, $v$.
							
							In $H^*(E)$ we have $s^2=0$, and thus the triple Massey product $\<s,s,s\>$ is defined.
							The preimages of $s^2$ under the differential of $E$ are of the form $h_{n-2}+\lambda s$ with $\lambda\in\kk$.
							We have
							\begin{equation*}
								s(h_{n-2}+\lambda_1s)+(h_{n-2}+\lambda_2)s=sh_{n-2}+(\lambda_1+\lambda_2)s^2+h_{n-2}s\neq 0
							\end{equation*}
							for all $\lambda_1,\lambda_2\in\kk$, so $\<s,s,s\>$ does not vanish.
							It follows that $E$ is not formal since all triple Massey products of a formal dg algebra vanish by \cite[Thm.~4]{BabenkoTaimanov}.\qedhere
					\end{enumerate}
				\end{proof}
				\begin{thm}
					\label{koszuldual}
					Let $A=\kk[x]/(x^{n+1})$ with $n>1$ and $|x|=k$, and let $P_\kk$ denote the graded projective resolution of $\kk$ from \cref{projres}.
					\begin{enumerate}
						\item There is a quasi-isomorphism of $\ints$-graded dg algebras
							\begin{equation*}
								E\to\End_{\C_\dg(\grrMod A)}^*(P_\kk),
							\end{equation*}
						\item $A^{!,\dg}$ is quasi-isomorphic to $\Tot^\Pi(E)$.
					\end{enumerate}
				\end{thm}
				\begin{proof}\leavevmode
					\begin{enumerate}
					\item By \cref{extalgebra} we have $H^*(\End_{\C_\dg(\grrMod A)}^*(P_\kk))=\Ext_{\grrMod A}^{*,*}(\kk,\kk)\cong\kk[s,t]/(s^2)$, and the proof provides an explicit representative for the generator $s$.
						For $0\leq i\leq n-2$, define $h_i\in\End_{\C_\dg(\grrMod A)}^*(P_\kk)$ as
						\begin{equation*}
							\begin{tikzcd}[column sep=tiny]
								\dots\arrow{r}
								&A\<-(2n+2)k\>\arrow{r}{x^n}\arrow{d}{x^i}
								&A\<-(n+2)k\>\arrow{r}{x}\arrow{d}{0}
								&A\<-(n+1)k\>\arrow{r}{x^n}\arrow{d}{x^i}
								&A\<-k\>\arrow{r}{x}\arrow{d}{0}
								&A\arrow{d}\\
								\dots\arrow{r}
								&A\<-(2n+2-i)k\>\arrow{r}{x}
								&A\<-(2n+1-i)k\>\arrow{r}{x^n}
								&A\<-(n+1-i)k\>\arrow{r}{x}
								&A\<-(n-i)k\>\arrow{r}
								&0\rlap{.}
							\end{tikzcd}
						\end{equation*}
						Note that the representative of $t$ from \cref{extalgebra} is given by $t=sh_0+h_0s$.
						
						It is easy to see that the morphisms $s$ and $h_i$ satisfy the relations of $E$, and also $d(s)=0$, $d(h_{n-2})=s^2$ and $d(h_i)=h_{i+1}s+sh_{i+1}$ for $i<n-2$.
						Hence sending the generators of $E$ to the morphisms $s$ and $h_i$ defines a morphism of $\ints$-graded dg algebras, and this is a quasi-isomorphism by \cref{extalgebra} and \cref{dgalgcohomology}.
					\item This is immediate from 1), \cref{projresdiagfin} and \cref{dgvsgradedkoszul}.\qedhere
					\end{enumerate}
				\end{proof}
				\begin{rem}
					The $\ints$-graded dg algebra $E$ is diagonal-finite, and thus $\Tot^\Pi(E)=\Tot^\amalg(E)$.
				\end{rem}
			\subsubsection{The case \texorpdfstring{$n=1$, $k\neq 1$}{n=1, k!=1}}
				The overall argument for $n=1$ and $k\neq 1$ is almost the same as for $n>1$, but no extra calculations are required.
				\begin{prop}
					\label{koszuldualn1}
					Let $A=\kk[x]/(x^2)$ with $|x|=k\neq 1$, and let $P_\kk$ denote the graded projective resolution of $\kk$ from \cref{projres}.
					\begin{enumerate}
						\item There is a quasi-isomorphism of $\ints$-graded dg algebras
							\begin{equation*}
								\kk[s]\to\End_{\C_\dg(\grrMod A)}^*(P_\kk),
							\end{equation*}
							with the bigrading $|s|=(1,-k)$ and differential $d(s)=0$.
						\item $A^{!,\dg}$ is quasi-isomorphic to $\kk[s]$ with $|s|=1-k$ and $d(s)=0$.
					\end{enumerate}
				\end{prop}
				\begin{proof}\leavevmode
					\begin{enumerate}
						\item By \cref{extalgebra} we have $H^*(\End_{\C_\dg(\grrMod A)}^*(P_\kk))\cong\kk[s]$ with $|s|=(1,-k)$.
							The proof of \cref{extalgebra} provides an explicit representative for $s$, which yields a dg algebra quasi-isomorphism $\kk[s]\to\End_{\C_\dg(\grrMod A)}^*(P_\kk)$.
						\item This is immediate from 1), \cref{projresdiagfin} and \cref{dgvsgradedkoszul}.
							Note that $\kk[s]$ with $|s|=(1,-k)$ is diagonal-finite as $k\neq 1$, and thus $\Tot^\Pi(\kk[s])=\Tot^\amalg(\kk[s])=\kk[s]$ with $|s|=1-k$.\qedhere
					\end{enumerate}
				\end{proof}
			\subsubsection{The case \texorpdfstring{$n=1$, $k=1$}{n=1, k=1}}
				If $n=1$ and $k=1$, we cannot apply \cref{dgvsgradedkoszul} since the projective resolution of $\kk$ is not diagonal-finite.
				Hence we have to compute the dg Koszul dual directly.
				\begin{prop}
					\label{koszuldualn1k1}
					Let $A=\kk[x]/(x^2)$ with $|x|=1$.
					Then $A^{!,\dg}$ is quasi-isomorphic to $\kk[\![s]\!]$ with $|s|=0$ and $d(s)=0$.
				\end{prop}
				\begin{proof}
					By definition and \cref{kprojres} we have $A^{!,\dg}=\End_{\C_\dg(A)}(\Tot^\amalg(P_\kk))$, where $P_\kk$ denotes the graded projective resolution of $\kk$ from \cref{projres}.
					As a graded $A$-module, $\Tot^\amalg(P_\kk)\cong\coprod_{i\leq 0} A$ where the $i$-th copy of $A$ is the (shifted) $A\<i\>$ from the degree~$i$ part of $P_\kk$.
					Note that morphisms of $A$-modules $\coprod_{i\leq 0} A\to\coprod_{i\leq 0} A$ can be written as matrices with entries in $A$, extending infinitely to the right and down (with rows and columns indexed by $0,-1,-2,\dots$) and only finitely many non-zero entries in each column.
					The composition is then given by the usual matrix multiplication.
					In this notation, the differential of $\Tot^\amalg(P_\kk)$ is given by the infinite Jordan block with entries $x$ above the main diagonal.
					
					The dg algebra $\End_{\C_\dg(A)}(\Tot^\amalg(P_\kk))$ is concentrated in degrees $0$ and $1$.
					Its differential is given by the (graded) commutator with the differential of $\Tot^\amalg(P_\kk)$.
					A straightforward calculation shows that the cocycles of degree~$0$ are those matrices with entries in $A^0=\kk$ that are constant along diagonals and have $0$ below the main diagonal, and that all degree~$1$ elements are coboundaries.
					It follows that a quasi-isomorphism $\kk[\![s]\!]\to\End_{\C_\dg(A)}(\Tot^\amalg(P_\kk))$ is given by sending a power series $\sum_{r\geq 0} a_rs^r$ to the matrix $\begin{smallpmatrix}a_0&a_1&a_2&\dots\\0&a_0&a_1&\sddots\\0&0&a_0&\sddots\\\svdots&\sddots&\sddots&\sddots\end{smallpmatrix}$.
				\end{proof}
				\begin{rem}
					The proof of \cref{dgvsgradedext} fails for $n=1$, $k=1$: in this case the coproduct $\coprod_{p\in\ints} P_\kk[p]\<-p\>$ is not isomorphic to the product $\prod_{p\in\ints} P_\kk[p]\<-p\>$.
					However, we nevertheless have
					\begin{equation*}
						\begin{split}
							H^*(\End_{\C_\dg(A)}(\Tot^\amalg(P_\kk)))&\cong H^*(\Tot^\Pi(\End_{\C_\dg(\grrMod A)}^*(P_\kk)))\\
							&=\Tot^\Pi(\Ext_{\grrMod A}^{*,*}(\kk,\kk)).
						\end{split}
					\end{equation*}
				\end{rem}
		\subsection{The Koszul duality equivalence}
			For a proper non-positive dg algebra $A$, contravariant dg Koszul duality provides an equivalence $\perf(A)^\op\to\Dfd((A^{!,\dg})^\op)$.
			To obtain similar equivalences if $A$ is proper positive, the following property introduced in \cite[Def.~3.23]{FushimiDgKoszul} is required.
			There are several other equivalent characterizations, see \cite[Thm.~3.14]{FushimiDgKoszul}.
			
			Recall that the perfect derived category of a locally finite positive dg algebra $A$ admits a standard \textbf{t}-structure with $\heart=\extclos(A)$ \cite[Thm.~7.1]{KellerNicolas}.
			\begin{dfn}
				A locally finite positive dg algebra $A$ is \emph{pvd-finite} if the heart of the standard \textbf{t}-structure on $\perf(A)$ has an injective cogenerator.
			\end{dfn}
			\begin{lemma}
				\label{pvdfinite}
				Let $n>0$ and $k>0$.
				The dg algebra $A=\kk[x]/(x^{n+1})$ with $|x|=k$ is pvd-finite unless $n=1$ and $k=1$.
			\end{lemma}
			\begin{proof}
				This follows directly from (the proof of) \cref{indecinstdheart}:
				for $k>1$, the heart of the standard \textbf{t}-structure is semisimple, and hence the unique simple object $A\in\heart$ is an injective cogenerator.
				For $k=1$ and $n>1$ we have $\heart=\add(\{A,X_2\})$ where $X_2$ is a length~$2$ extension of $A$ with itself, and $X_2$ is an injective cogenerator.
				For $k=1$ and $n=1$ we have $\heart=\add(\{X_i\mid i\geq 1\})$ where $X_i$ is uniserial with $i$ many simple subquotients $A$, and none of the $X_i$ is injective.
			\end{proof}
			It follows that dg Koszul duality yields an alternative model of the triangulated category generated by a $\proj^n[k]$-object.
			For $n=1$ this is stated implicitly in \cite[\S 2]{KellerYangZhou}, cf.~also \cite[\S 1.1]{HolmJorgensenYang}.
			\begin{cor}
				\label{koszulmodel}
				Let $A=\kk[x]/(x^{n+1})$ with $|x|=k\in\ints$, and assume that $n>0$.
				\begin{enumerate}
					\item If $n\neq 1$ or $k\neq 1$, then $\perf(A)\cong\Dfd((A^{!,\dg})^\op)$.
						Moreover $\Dfd((A^{!,\dg})^\op)$ is generated by the $\proj^n[k]$-object $\kk$ with trivial action of $(A^{!,\dg})^\op$.
					\item If $n=1$ and $k=1$, then $\perf(A)\cong\thick(\kk[\![t]\!]/(t))\subseteq\D(\kk[\![t]\!])$ and moreover $\perf(A)\cong\thick(\kk[t]/(t))\subseteq\Dfd(\kk[t])$.
						The dg modules $\kk[\![t]\!]/(t)$ and $\kk[t]/(t)$ are $\proj^1[1]$-objects in the thick subcategories of $\D(\kk[\![t]])$ respectively $\D(\kk[t])$ that they generate.
				\end{enumerate}
			\end{cor}
			\begin{proof}\leavevmode
				\begin{enumerate}
					\item By \cite[Thm.~4.4]{FushimiDgKoszul}, the functor
						\begin{equation*}
							\RHom_A(-,\kk)\: \perf(A)^\op\to\Dfd((A^{!,\dg})^\op)
						\end{equation*}
						is an equivalence if $A$ is either locally finite non-positive, or locally finite positive and pvd-finite.
						By \cref{pvdfinite}, for $n>0$ and $k>0$ (i.e.~for $A$ positive) $A$ is pvd-finite unless $n=1$ and $k=1$.
						Since $A$ is proper, $DA\cong A[nk]$ as dg $A$-module, and $A$ is commutative, we have $\perf(A)^\op\cong\perf(A^\op)\cong\perf(A)$.
						The Koszul duality functor takes $A$ to $\kk$, and thus $\kk$ is a $\proj^n[k]$-object generating $\Dfd((A^{!,\dg})^\op)$.
					\item It is easy to see (by computing their projective resolutions) that $\kk[\![t]\!]/(t)$ in $\D(\kk[\![t]\!])$ and $\kk[t]/(t)$ in $\Dfd(\kk[t])$ are $\proj^1[1]$-like objects.
						Hence the subcategories $\thick(\kk[\![t]\!]/(t))\subseteq\D(\kk[\![t]\!])$ and $\thick(\kk[t]/(t))\subseteq\Dfd(\kk[t])$ are triangulated categories generated by a $\proj^1[1]$-like object, and thus equivalent to $\perf(A)$ by \cref{moritamodel}.
						A $\proj$-like object is a $\proj$-object in the thick subcategory that it generates by \cref{plikecy}.\qedhere
				\end{enumerate}
			\end{proof}
	\printbibliography
\end{document}